\documentclass[oneside,11pt,reqno]{amsart}
\usepackage{amssymb,amsmath,amsthm,bbm,enumerate,mdwlist,url,multirow,hyperref,amsthm,mathabx}
\usepackage[pdftex]{graphicx}
\usepackage[shortlabels]{enumitem}

\theoremstyle{definition}
\newtheorem{definition}{Definition}
\theoremstyle{theorem}
\newtheorem{proposition}[definition]{Proposition}

\numberwithin{equation}{section}
\numberwithin{definition}{section}
\theoremstyle{remark}
\newtheorem{remark}[definition]{Remark}
\newtheorem{example}[definition]{Example}
\newtheorem{question}[definition]{Question}

\def\PP{\mathsf P}
\def\EE{\mathsf E}
\def\QQ{\mathsf Q}
\def\BB{\mathcal B}
\def\GG{\mathcal G}
\def\FF{\mathcal F}
\def\CC{\mathcal C}
\def\HH{\mathcal H}
\def\FFF{\mathfrak F}
\def\NN{\mathcal N}
\def\II{\mathcal I}
\def\LL{ \mathrm{L}}
\def\AA{\mathcal A}

\def\Ber{\mathrm{Ber}}
\newcommand\xqed[1]{%
  \leavevmode\unskip\penalty9999 \hbox{}\nobreak\hfill
  \quad\hbox{#1}}
\newcommand\finish{\xqed{$\diamond$}}
\begin{document}
\title{A couple of remarks on the convergence of $\sigma$-fields on probability spaces}

\author{Matija Vidmar}
\address{Department of Mathematics, University of Ljubljana and Institute of Mathematics, Physics and Mechanics, Slovenia}
\email{matija.vidmar@fmf.uni-lj.si}

\begin{abstract}
The following modes of convergence of sub-$\sigma$-fields on a given probability space have been studied in the literature: weak convergence, strong convergence, convergence with respect to the Hausdorff metric, almost-sure convergence, set-theoretic convergence, monotone convergence. It is noted that all preserve independence in the limit, and all are invariant under passage to an equivalent probability measure. Partial results for the case of operator-norm convergence obtain. 
\end{abstract}




\maketitle
\section{Introduction}

Fix a (not necessarily complete) probability space $(\Omega,\FF,\PP)$ and set $\NN:=\PP^{-1}(\{0\})=\{F\in \FF:\PP(F)=0\}$. For a sub-$\sigma$-field $\AA$ of $\FF$, $\overline{\AA}^\PP:=\AA\lor\sigma(\NN)=\sigma(\AA\cup\NN)$, the $\PP$-completion of $\AA$. A \emph{$\sigma$-subfield} means a $\PP$-complete sub-$\sigma$-field of $\FF$, i.e. a sub-$\sigma$-field of $\FF$ that is equal to its $\PP$-completion. The collection of all $\sigma$-subfields is denoted $\FFF$. 
For an $\FF/\mathcal{B}([-\infty,\infty])$-measurable map $f$ satisfying $\int  f^+d\PP\land \int f^-d\PP<\infty$, $\PP f:=\EE^\PP[f]$; if further $\AA\in \FFF$, then $\PP_\AA f:=\EE^\PP[f\vert \AA]$, the conditional expectation of $f$ w.r.t. $\AA$ under $\PP$.

Recall now that for a sequence  $(\BB_n)_{n\in \mathbb{N}_0}\subset \FFF$, classical martingale theory gives, for any $f\in \LL^1(\PP)$, the convergence $\PP_{\BB_n}f\to \PP_{\BB_0}f$ in $\LL^1(\PP)$ and $\PP$-a.s. as $n\to\infty$, provided one has
\begin{enumerate}[label=(MC),ref=(MC)]
\item \label{MC} \emph{Monotone convergence}. $\BB_n\subset \BB_{n+1}$ for all $n\in \mathbb{N}$ and $\BB_0=\lor_{n\in \mathbb{N}}\BB_n$, or $\BB_n\supset \BB_{n+1}$ for all $n\in \mathbb{N}$ and $\BB_0=\cap_{n\in \mathbb{N}}\BB_n$ \cite[Theorem~6.23]{kallenberg}.
\end{enumerate}
Generalizing/complementing this monotone convergence, the following ways of making precise the concept of convergence of a sequence of $\sigma$-subfields $(\BB_n)_{n\in \mathbb{N}}$ to a $\sigma$-subfield $\BB_0$ (under $\PP$), have been proposed and studied in the literature (among others; all the convergences are as $n\to\infty$):
\begin{enumerate}[label=(WC),ref=(WC)]
\item\label{WC} \emph{Weak convergence}. $\BB_n$ converges weakly to $\BB_0$  if $\PP_{\BB_n}\mathbbm{1}_A\to \PP_{\BB_0}\mathbbm{1}_A=\mathbbm{1}_A$ in $\PP$-probability for every $A\in \BB_0$ \cite{coquet,alonso-brambila}.
\end{enumerate}
\begin{enumerate}[label=(SC),ref=(SC)]
\item\label{SC} \emph{Strong convergence}. $\BB_n$ converges strongly to $\BB_0$  if $\PP_{\BB_n}\mathbbm{1}_A\to  \PP_{\BB_0}\mathbbm{1}_A$ in $\PP$-probability for every $A\in \FF$ \cite{kudo,piccinini,artstein,alonso-brambila} \cite[Problem~IV.3.2]{neveu-book} \cite[Section~2]{tsirelson} \cite[Section VIII.2]{hu}.
\end{enumerate}
\begin{enumerate}[label=(HC),ref=(HC)]
\item\label{HC} \emph{Hausdorff convergence}. $\BB_n$ converges to $\BB_0$  w.r.t. the Hausdorff metric if $D(\BB_n,\BB_0)\to 0$, where for $\AA\in \FFF$ and $\BB\in \FFF$, $D(\AA,\BB):=\rho(\AA,\BB)+\rho(\BB,\AA)$ with $\rho(\AA,\BB):=\sup_{A\in \AA}\inf_{B\in \BB}\PP(A\triangle B)$ \cite{boylan,rogge,mukerjee,neveu,zandt,landers} \cite[Section VIII.2]{hu}.
\end{enumerate}
\begin{enumerate}[label=(STC),ref=(STC)]
\item\label{STC} \emph{Set-theoretic convergence}. $\BB_n$ converges to $\BB_0$  in the set-theoretic sense if $\liminf_{n\to\infty} \BB_n:=\lor_{n\geq 1}\cap_{k\geq n}\BB_k=\BB_0=\cap_{n\geq 1}\lor_{k\geq n}\BB_k=:\limsup_{n\to\infty}\BB_n$  \cite{fetter,alonso} \cite[Problem~II.6]{rao}.
\end{enumerate}
\begin{enumerate}[label=(ASC),ref=(ASC)]
\item\label{ASC} \emph{Almost-sure convergence}. $\BB_n$ converges to $\BB_0$  in the almost-sure sense if $\PP$-a.s. $\PP_{\BB_n}f\to \PP_{\BB_0}f$ for any $f\in \LL^1(\PP)$  \cite{alonso}.
\end{enumerate}
And, for $\{p,q\}\subset [1,\infty]$, $q\leq p$:
\begin{enumerate}[label=(ONC$_{p}^q$),ref=(ONC$_{p}^q$),itemindent=0.25cm]
\item\label{ONC} \emph{Operator-norm convergence}. $\BB_n$ converges to $\BB_0$  in the operator-norm sense if $\PP_{\BB_n}\to \PP_{\BB_0}$ in the operator norm $\Vert\cdot\Vert_{\LL^p\to\LL^q}$ when viewed as  mappings between the (real) normed spaces $(\LL^p(\PP),\Vert\cdot\Vert_{\LL^p(\PP)})$ and $(\LL^q(\PP),\Vert\cdot\Vert_{\LL^q(\PP)})$  \cite{rogge}. 
\end{enumerate}
Beyond the obvious relevance of these convergence types to the issue of continuity of conditional expectations w.r.t. the conditioning $\sigma$-field, we note applications in statistics \cite{kudo-app,maxwell}, studying closeness and convergence of information \cite{zandt} \cite[Section VIII.2]{hu}, to the theory of noises \cite{tsirelson} (see also the references therein).

As for our contribution, we shall demonstrate two desirable properties shared by (essentially) all these modes of convergence. First, they all preserve independence in the limit --- and this claim generalizes to conditional independence, save for \ref{WC} --- see Section~\ref{section:independence}  (in particular Remark~\ref{remark:preservation} for the precise meaning of preservation of independence in the limit). Second, 
excepting (perhaps) only \ref{ONC} when $p=q$, all are invariant under passage to an equivalent probability measure (the latter is trivial for \ref{MC} and \ref{STC}, but not obvious for the others) --- see Section~\ref{section:invariance}. Given that $\sigma$-subfields are often interpreted as bodies of information, and hence convergence of these as a convergence of information, it is certainly note-worthy that all these types of convergence do in fact depend on the underlying probability measure $\PP$ only via $\NN=\PP^{-1}(\{0\})$ (which, short of dispensing with the probability measure altogether, is surely the best we can hope for). Likewise, independence is a fundamental probabilistic property -- its preservation in the limit of $\sigma$-fields deserves to be made explicit. We will indeed see in relation to this, that the simultaneous consideration of the various convergence types enunciated above allows for a great economy of argument. 

We will also show \emph{en passant} that \ref{ONC} with $q<p$ is equivalent to \ref{HC}, while the case $p=q=1$ or $p=q=\infty$ is vacuous (apart from the trivial case of $(\BB_n)_{n\in\mathbb{N}}$ ultimately constant), but the case $p=q\in (1,\infty)$ is not (Section~\ref{sect:preliminaries}).

Finally, in terms of what has been noted in the literature in connexion to this thus far:
\begin{itemize}[leftmargin=0cm,itemindent=.5cm]
\item \cite[Corollary~3.6]{tsirelson} gives, assuming $\LL^2(\PP)$ is separable, preservation of \emph{pairwise} independence in the strong limit. (Were the join operation $\lor$ (sequentially) continuous under \ref{SC}, then preservation of independence would be an essentially immediate corollary. But it is not, see e.g. the example of \cite[Section~1.2]{tsirelson}.) 
\item \cite[Theorem~VIII.2.23]{hu} gives invariance of \ref{HC} under passage to a ``uniformly absolutely continuous'' (see  \cite[Definition~VIII.2.22]{hu}) probability measure. 
\item \cite[Theorem~VIII.2.40]{hu} gives invariance of \ref{SC} under passage to an equivalent probability measure, assuming $\LL^1(\PP)$ is separable. 
\end{itemize}
(Note that for $r\in [1,\infty)$, the separability of $\LL^r(\PP)$ is equivalent to $\FF=\overline{\sigma(\AA)}^\PP=\sigma(\AA\lor\NN)$ for some denumerable $\AA\subset \FF$.) 

\section{Preliminaries}\label{sect:preliminaries}
We gather some relevant results scattered in the literature and make some observations. Let $(\BB_n)_{n\in \mathbb{N}_0}$ be a sequence of $\sigma$-subfields. 

\begin{enumerate}[leftmargin=0cm,itemindent=.5cm,label=\textbf{(\arabic*)},ref=\arabic*]
\item\label{preliminaries:Hausdorff} \textbf{Hausdorff metric}. Thanks to the insistence on the $\PP$-completeness of the $\sigma$-subfields, $D$ is a metric on $\FFF$ \cite[Theorem~1 \& Corollary~1]{boylan}. By taking $\lor$ instead of $+$ in its definition (see \ref{HC}), one obtains an equivalent metric $\delta\leq 1\land D$, which is the restriction to $\FFF\times \FFF$ of the usual Hausdorff distance on closed subsets of $\FF$, associated to the pseudometric $\FF\times \FF\ni (A,B)\mapsto \PP(A\triangle B)\in [0,1]$. 
\item \label{preliminaries:order} \textbf{Implications and non-implications between the various convergence types}. 

As already observed in the Introduction, \ref{MC} $\Rightarrow$ \ref{ASC} (and, of course, \ref{MC} $\Rightarrow$  \ref{STC}). 

\ref{SC} $\Rightarrow$ \ref{WC} trivially, but not the other way around  --- not even when  \ref{WC} is to the largest $\sigma$-subfield $\BB_0$ for which it holds, cf. \eqref{preliminaries:weak}, second bullet point, below ---  \cite[Example~3.1]{kudo}. 

\ref{HC} $\Rightarrow$ \ref{SC} \cite[Theorem~4]{boylan} and \ref{STC} $\Rightarrow$ \ref{SC} \cite{fetter}, though neither conversely \cite[Example~4.4]{alonso-brambila}; \ref{ASC} $\Rightarrow$ \ref{SC} trivially, again the converse fails \cite[Example~3.5]{piccinini}. 

Of \ref{STC}, \ref{HC}, \ref{ASC} no one implies another: \ref{STC} $\nRightarrow$ \ref{ASC} \cite{alonso};  \ref{ASC} $\nRightarrow$ \ref{STC} \cite[Example 4.1]{alonso-brambila};  \ref{STC} $\nRightarrow$ \ref{HC} \cite[Example 4.2]{alonso-brambila};  \ref{HC} $\nRightarrow$ \ref{STC} \cite[Example 4.3]{alonso-brambila};  \ref{HC} $\nRightarrow$ \ref{ASC} \cite[penultimate paragraph]{boylan}; \ref{MC} (hence \ref{ASC}, \ref{STC}) $\nRightarrow$ \ref{HC} \cite[Example 4.2]{alonso-brambila}. 

\ref{ONC} $\Rightarrow$  \ref{HC}.
\begin{proof}
For $\sigma$-subfields $\AA$ and $\BB$, and for $A\in \AA$, $\inf_{B\in \BB}\PP(A\triangle B)=\inf_{B\in \BB}\PP \vert \mathbbm{1}_A-\mathbbm{1}_B\vert\leq \PP \vert \mathbbm{1}_A-\mathbbm{1}_{\{\PP_\BB\mathbbm{1}_A>1/2\}}\vert$. Now if  $q<\infty$, we obtain $\inf_{B\in \BB}\PP(A\triangle B)\leq 2^q\PP\vert \mathbbm{1}_A-\PP_\BB\mathbbm{1}_A\vert^q=2^q\PP\vert \PP_\AA\mathbbm{1}_A-\PP_\BB\mathbbm{1}_A\vert^q\leq 2^q\Vert \PP_\AA-\PP_\BB\Vert_{\LL^p\to\LL^q}^q$, where we have used $\vert \mathbbm{1}_A-\mathbbm{1}_{\{\PP_\BB\mathbbm{1}_A>1/2\}}\vert\leq 2^q\vert \mathbbm{1}_A-\PP_\BB\mathbbm{1}_A\vert^q$. When $q=\infty$, we have simply $ \inf_{B\in \BB}\PP(A\triangle B)\leq 2\PP\vert \PP_\AA\mathbbm{1}_A-\PP_\BB\mathbbm{1}_A\vert\leq 2\Vert \PP_\AA-\PP_\BB\Vert_{\LL^\infty\to\LL^\infty}$, since $\vert \PP_\AA\mathbbm{1}_A-\PP_\BB\mathbbm{1}_A\vert\leq \Vert \PP_\AA-\PP_\BB\Vert_{\LL^\infty\to\LL^\infty}$.
\end{proof}

Finally, \ref{HC} $\Rightarrow$  \ref{ONC}, assuming that $p>q$.
\begin{proof}
Recall the metric $\delta$ from \eqref{preliminaries:Hausdorff} and $\rho$ from \ref{HC}. We quote the following two results from the literature: 
\begin{enumerate}[(a)]
\item\label{help:a} Let $a\in (0,\infty)$, $r\in [1,\infty)$, $H\subset \LL^r(\PP)$ and define $\delta_{H,r}(a):=\sup\{\Vert f\mathbbm{1}_{\{\vert f\vert>a\}}\Vert_{\LL^r(\PP)}:f\in  H\}$. Then, for $\sigma$-subfields $\AA$ and $\BB$  satisfying $\AA\subset \BB$, one has the inequality $$\sup\{\Vert \PP_{\BB}f-\PP_{\AA}f\Vert_{\LL^r(\PP)}:f\in H\}\leq C_ra[\delta(\AA,\BB)(1-\delta(\AA,\BB))]^{1/r}+2\delta_{H,r}(a),$$ where $C_r=2\cdot 2^{1/r}$ if $r<2$ and $C_r=2$ if $r\geq 2$. \cite[Theorem~4, items (i) \& (ii)]{rogge}
\item\label{help:b} $\rho(\AA\lor\BB,\BB)\leq 4\rho(\AA,\BB)$ for $\sigma$-subfields $\AA$ and $\BB$.  \cite[Corollary~4]{landers}
\end{enumerate}
Set $H:=\{f\in \LL^p(\PP):\Vert f\Vert_{\LL^p(\PP)}\leq 1\}$. For any $f\in H$, by the triangle inequality, $$\Vert \PP_{\BB_n}f-\PP_{\BB_0}f\Vert_{\LL^q(\PP)}\leq \Vert \PP_{\BB_n}f-\PP_{\BB_0\lor\BB_n}f\Vert_{\LL^q(\PP)}+\Vert \PP_{\BB_n\lor \BB_0}f-\PP_{\BB_0}f\Vert_{\LL^q(\PP)}.$$ Then, by \ref{help:a}, for any $a\in (0,\infty)$, $$\sup\{\Vert \PP_{\BB_n}f-\PP_{\BB_0}f\Vert_{\LL^q(\PP)}:f\in H\}$$ 
$$\leq 2\cdot 2^{1/q}a[\delta(\BB_n,\BB_0\lor \BB_n)(1-\delta(\BB_n,\BB_0\lor\BB_n))]^{1/q}+2\delta_{H,q}(a)+$$
$$ 2\cdot 2^{1/q}a[\delta(\BB_n\lor\BB_0,\BB_0)(1-\delta(\BB_n\lor \BB_0,\BB_0))]^{1/q}+2\delta_{H,q}(a)$$ (where the notation $\delta_{H,q}(a)$ is that of \ref{help:a} above), which by \ref{help:b} is (note that for $\sigma$-subfields $\AA\supset\BB$, $\delta(\AA,\BB)=\delta(\BB,\AA)=\rho(\AA,\BB)$) 
$$\leq 
4\cdot 8^{1/q}a\delta(\BB_n,\BB_0)^{1/q}+4\delta_{H,q}(a).
$$ Since $\delta_{H,q}(a)\leq a^{-(\frac{p}{q}-1)}$ when $p<\infty$ and $\delta_{H,q}(a)\leq \mathbbm{1}_{[0,1)}(a)$ when $p=\infty$, it follows that $\limsup_{n\to\infty}\sup\{\Vert \PP_{\BB_n}f-\PP_{\BB_0}f\Vert_{\LL^q(\PP)}:f\in H\}\leq 4\delta_{H,q}(a)\to 0$ as $a\to\infty$. Hence $\lim_{n\to\infty}\sup\{\Vert \PP_{\BB_n}f-\PP_{\BB_0}f\Vert_{\LL^q(\PP)}:f\in H\}=0$, which is the desired operator norm convergence. 
\end{proof}
(The last two implications, in the case $q=1$, $p=\infty$ can be found e.g. in \cite[Theorem VIII.2.21]{hu}.)

Up to trivial corollaries, this exhausts the mutual implications and non-implications of the convergence types  \ref{ONC} for $p>q$, \ref{MC}, \ref{HC}, \ref{STC}, \ref{ASC}, \ref{SC}, \ref{WC} (Figure~\ref{figure:impli}).
\begin{figure}
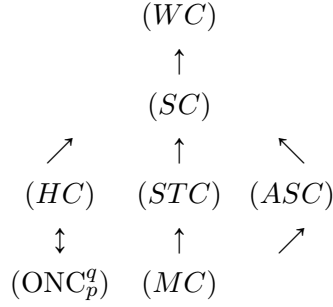

$$
\begin{array}{ccc}
&\ref{WC}&\\
&\uparrow&\\
&\ref{SC}&\\
\nearrow &\uparrow& \nwarrow\\
\ref{HC}&\ref{STC}&\ref{ASC}\\
\updownarrow &\uparrow& \nearrow\\
$\ref{ONC}$&\ref{MC}&
\end{array}
$$
\caption{Implications between the various types of convergence (with $q<p$ for the case of \ref{ONC}); absence of an (by transitivity implied) arrow means the implication fails in general. 
}\label{figure:impli}
\end{figure}

\item \textbf{Uniqueness of limits}. Excepting \ref{WC}, the limits are unique. Indeed, the results of e.g. \cite{kudo} imply uniqueness of the limit in the case of \ref{SC}.

\item \label{preliminaries:weak}\textbf{Weak covergence}. 

$\bullet$ For any $p\in [1,\infty)$, $\BB_n\to \BB_0$ weakly iff $\PP_{\BB_n}f\to f$ in $\LL^p(\PP)$ for every $f\in \LL^p(\PP\vert_{\BB_0})$ (i.e. for every $f\in \LL^p(\PP)$ that is $\BB_0$ measurable). 
\begin{proof}
By linearity, for sure $\PP_{\BB_n}f\to  \PP_{\BB_0}f$ in $\PP$-probability for any bounded simple $\BB_0$-measurable $f$. Now let $f\in \LL^p(\PP\vert_{\BB_0})$; $\delta>0$. The simple functions being dense in $\LL^p(\PP)$, there exists a simple $\BB_0$-measurable $f^\delta$ for which $\PP\vert f-f^\delta\vert^p<\delta$. Then it follows from the decomposition $$\PP_{\BB_n}f-\PP_{\BB_0}f=\PP_{\BB_n}(f-f^\delta)+\PP_{\BB_n}f^\delta-\PP_{\BB_0}f^\delta+\PP_{\BB_0}(f^\delta-f),$$ from the elementary estimate $\vert x+y\vert^p\leq 2^{p-1}(\vert x\vert^p+\vert y\vert^p)$ for $\{x,y\}\subset \mathbb{R}$, from conditional Jensen's inequality, finally from the fact that boundedness implies uniform integrability (hence coupled with convergence in $\PP$-probability, $\LL^1(\PP)$ convergence), that $\limsup_{n\to\infty}\PP\vert \PP_{\BB_n}f-\PP_{\BB_0}f\vert^p\leq C\delta$, for some constant $C\in (0,\infty)$ depending only on $p$. Let $\delta\downarrow 0$.
\end{proof}
$\bullet$ Then, according to \cite[Lemmas~1.1 and~1.3]{alonso-brambila}, $\BB_n\to \BB_0$ weakly as $n\to \infty$ iff $\BB_0\subset \BB_\PP,$ where 
\begin{equation*}
\BB_\PP:=\{A\in \FF:\lim_{n\to\infty}\inf_{B\in \BB_n}\PP(A\triangle B)=0\}.
\end{equation*}
The $\sigma$-subfield $\BB_\PP$ coincides with the $\PP\text{-}\liminf_{n\to\infty}\BB_n$ of \cite{kudo}, see \cite[Theorem~3.2]{kudo}.

$\bullet$ The join (sup) operation $\lor$ is sequentially continuous under weak convergence \cite[Proposition~2.3]{coquet} (but the meet (inf) $\land$ is not \cite[Proposition~2.1]{coquet}). It means that for sequences $(\AA_n)_{n\in \mathbb{N}}$ and $(\BB_n)_{n\in \mathbb{N}}$ in $\FFF$, if $\AA_n\to\AA_0$ and $\BB_n\to\BB_0$ weakly, then also $\AA_n\lor\BB_n\to\AA_0\lor \BB_0$ weakly (but in general this fails if $\cap$ replaces $\lor$). 
\item\label{preliminaries:strong} \textbf{Strong convergence}. With an analogous justification to the one in \eqref{preliminaries:weak}, for any $p\in [1,\infty)$, $\BB_n\to \BB_0$ strongly iff $\PP_{\BB_n}f\to \PP_{\BB_0}f$ in $\LL^p(\PP)$ for every $f\in \LL^p(\PP)$. Thus strong convergence is nothing but the strong operator convergence of the conditional expectation operators in the spaces $\LL^p(\PP)$, $p\in [1,\infty)$ (but not in $\LL^\infty(\PP)$; the latter fails even for monotone increasing sequences, see \cite[final paragraph]{fetter}). 

\item\label{operator} \textbf{Operator convergence $\LL^p\to\LL^p$}. Convergence in the operator norm $\Vert\cdot\Vert_{\LL^p\to\LL^p}$, $p\in [1,\infty]$, appears elusive.
\begin{itemize}[leftmargin=0cm,itemindent=.5cm]
\item For one, \ref{ONC} of not-ultimately-constant sequences of $\sigma$-subfields fails always when $p=q=1$ or when $p=q=\infty$. Indeed, if $\AA$ and $\BB$ are two distinct $\sigma$-subfields, then $\Vert\PP_\AA-\PP_\BB\Vert_{\LL^1\to\LL^1}\geq 1$ and $\Vert \PP_\AA-\PP_\BB\Vert_{\LL^\infty\to\LL^\infty}\geq 1/2$. To see this, let $B\in \BB\backslash \AA$. Set first $f=\mathbbm{1}_B-\PP_\AA\mathbbm{1}_B$, so that $0\ne f\in \LL^1(\PP)$. Then $\PP$-a.s., $\PP_\AA f=0$, while $f$ (hence $\PP_\BB f$) is $\geq 0$ on $B$ and $\leq 0$ on $\Omega\backslash B$. It follows that $\PP\vert \PP_\BB f-\PP_\AA f \vert=\PP\vert \PP_\BB f\vert=\PP[(\PP_\BB f)\mathbbm{1}_B-(\PP_\BB f)\mathbbm{1}_{\Omega\backslash B}]=\PP f\mathbbm{1}_B-\PP f\mathbbm{1}_{\Omega \backslash B}=\PP \vert f\vert$, viz. $\Vert\PP_\AA-\PP_\BB\Vert_{\LL^1\to\LL^1}\geq 1$. Set now $f=\mathbbm{1}_B$, noting that $\Vert f\Vert_{\LL^\infty(\PP)}=1$. Then $\PP(B\triangle \{\PP_\AA\mathbbm{1}_B>1/2\})>0$ so that either $\PP(B\backslash \{\PP_\AA\mathbbm{1}_B>1/2\})>0$ or $\PP(\{\PP_\AA\mathbbm{1}_B>1/2\}\backslash B)>0$, which coupled with the $\PP$-a.s. equality $\PP_\BB f=\mathbbm{1}_B$, yields $\Vert \PP_\AA-\PP_\BB\Vert_{\LL^\infty\to\LL^\infty}\geq 1/2$. 
%

\item If for infinitely many $n\in \mathbb{N}$, $\BB_n\subsetneq\BB_0$ or $\BB_0\subsetneq \BB_n$ --- or if $\BB_n\subsetneq \BB_m$ for arbitrarily large $n\in \mathbb{N}$ and $m\in \mathbb{N}$ --- then again $\PP_{\BB_n}$ does not converge to $\PP_{\BB_0}$ in the $\Vert\cdot\Vert_{\LL^p\to\LL^p}$ norm. For if $\AA$ is a $\sigma$-subfield that is strictly contained in the $\sigma$-subfield $\BB$, one can take $B\in \BB\backslash \AA$, set $f=\mathbbm{1}_B-\PP_\AA\mathbbm{1}_B$, which is then not $\PP$-a.s. equal to $0$, and finds that $\PP$-a.s. $\PP_{\BB}f-\PP_\AA f= f$.
In particular, one sees that 
\ref{ONC} simply precludes \ref{MC} of not-ultimately-constant sequences of $\sigma$-subfields altogether. 

\item If, for infinitely many $n\in \mathbb{N}$, $\BB_n$ contains a non-$\PP$-trivial event independent of $\BB_0$ or vice versa --- or if this obtains with $\BB_m$ in place of $\BB_0$ for arbitrarily large $m$ and $n$, --- then $\PP_{\BB_n}$ does not converge to $\PP_{\BB_0}$ in the $\Vert\cdot\Vert_{\LL^p\to\LL^p}$ norm. 
For if $A$ belongs to a $\sigma$-subfield $\AA$, is independent of a $\sigma$-subfield $\BB$, and has $\PP(A)\in (0,1)$, then taking $f=\PP(A)^{-1}\mathbbm{1}_A-(1-\PP(A))^{-1}\mathbbm{1}_{\Omega\backslash A}$, one has $\PP$-a.s. $ \PP_{\AA}f-\PP_\BB f= f-\PP f=f$, the latter not being $\PP$-a.s. equal to $0$.

\item Still, for $p\in (1,\infty)$, this convergence type is not vacuous with respect to not-ultimately-constant sequences of $\sigma$-subfields: 

\begin{example}\footnote{Due to J. Warren, private communication.}
Fix $p\in (1,\infty)$. Let for $r\in [0,1]$, $\Ber(r)$ denote the Bernoulli law (on the space $\{0,1\}$) with success parameter $r$: $\Ber(r)(\{0\})=1-r$ and $\Ber(r)(\{1\})=r$. Let  $\Omega=\{0,1\}^{\mathbb{N}_0}$; let $X_i$, $i\in \mathbb{N}_0$, be the canonical projections; $\FF$ the $\sigma$-field generated by them; $\PP=\Ber(1/2)\times \bigtimes_{n\in \mathbb{N}}\Ber(1/2^n)$. For $n\in \mathbb{N}_0$, set $Y_n:=X_0\lor X_n$ and then let $\BB_n=\overline{\sigma(Y_n)}^{\PP}$, the $\PP$-complete $\sigma$-subfield generated by $Y_n$. Note that by the tower property of conditional expectations and the independence of the $(X_m)_{m\in \mathbb{N}_0}$, for $n\in \mathbb{N}$, $\sup\{\PP \vert \PP_{\BB_n}f-\PP_{\BB_0}f\vert^p:f\in \LL^p(\PP), \Vert f\Vert_{\LL^p(\PP)}\leq 1\}=\sup\{\QQ^{1/2^n} \vert \QQ^{1/2^n}_{\AA_2}f-\QQ^{1/2^n}_{\AA_1}f\vert^p:f\in \LL^p(\QQ^{1/2^n}),\Vert f\Vert_{\LL^p(\QQ^{1/2^n})}\leq 1\}$, where $\QQ^\epsilon:=\Ber(1/2)\times\Ber(\epsilon)$ for $\epsilon\in [0,1]$, and where, with $Z_1:\{0,1\}\times\{0,1\}\to \{0,1\}$ and $Z_2:\{0,1\}\times\{0,1\}\to \{0,1\}$ being the projections onto the first and second coordinate respectively, $\AA_1:=\sigma(Z_1)$ and $\AA_2:=\sigma(Z_1\lor Z_2)$. 
Now fix $\epsilon\in [0,1]$. For $f\in \LL^1(\QQ^\epsilon)$, we have $$\QQ^\epsilon_{\AA_2}f=\mathbbm{1}_{\{(0,0)\}} f(0,0)+\mathbbm{1}_{\{(0,1),(1,0),(1,1)\}} \frac{f(0,1)\epsilon+ f(1,0)(1-\epsilon)+ f(1,1)\epsilon}{(1+\epsilon)},$$  whilst
 $$\QQ^\epsilon_{\AA_1}f=\mathbbm{1}_{\{(0,0),(0,1)\}}[f(0,0)(1-\epsilon)+f(0,1)\epsilon]+\mathbbm{1}_{\{(1,0),(1,1)\}}[f(1,0)(1-\epsilon)+f(1,1)\epsilon].$$ 

In view of the equality $\QQ^\epsilon \vert \QQ^\epsilon_{\AA_2}f-\QQ^\epsilon_{\AA_1}f\vert^p=\sum_{\omega\in \{0,1\}\times \{0,1\}} \QQ^\epsilon (\vert \QQ^\epsilon_{\AA_2}f-\QQ^\epsilon_{\AA_1}f\vert^p\mathbbm{1}_{\{\omega\}})$, and using the elementary estimate $(x+y)^r\leq 2^{r-1}(x^r+y^r)$ for $\{x,y\}\subset [0, \infty)$ and $r\in [1,\infty)$, it is now straightforward to see that $\QQ^\epsilon[\vert \QQ^\epsilon_{\AA_2}f-\QQ^\epsilon_{\AA_1}f\vert^p]\leq C_p\epsilon^{(p-1)\land 1}\QQ^\epsilon[ \vert f\vert^p]$ for some $C_p\in (0,\infty)$ depending only on $p$. It follows that $\PP_{\BB_n}\to \PP_{\BB_0}$ in the $\Vert\cdot\Vert_{\LL^p\to\LL^p}$ operator norm. \finish
\end{example}
\end{itemize}

\end{enumerate}

%
%

\section{Preservation of independence in the limit}\label{section:independence}

\begin{proposition}\label{proposition:preservation:strong}
Let $(\CC_n)_{n\in \mathbb{N}_0}$ be a sequence in $\FFF$ with $\CC_n\to \CC_0$ strongly as $n\to\infty$; $\II$ an arbitrary index set; finally $(\BB_n^i)_{(n,i)\in \mathbb{N}_0\times \II}$ a collection of $\sigma$-subfields of $\FF$ with 

$\BB_n^i\to \BB_0^i$ weakly as $n\to \infty$ for each $i\in \II$ 

\noindent and with

the family $(\BB_n^i)_{i\in \II}$ conditionally independent given $\CC_n$ for each $n\in \mathbb{N}$.

\noindent Then the family $(\BB_0^i)_{i\in \II}$ is conditionally independent given $\CC_0$. 
\end{proposition}
\begin{remark}\label{remark:preservation}
One says conditional independence is preserved in the limit, under a convergence type (iC), when the statement of Proposition~\ref{proposition:preservation:strong} prevails for any choice of the $\CC_n$ and the $\BB_n^i$s, and with the words ``strongly'' and ``weakly'' replaced by ``for the convergence type (iC)'' therein. This specializes to (unconditional) independence when $\CC_n$ is $\PP$-trivial for every $n\in \mathbb{N}_0$. By \eqref{preliminaries:order} it follows that (unconditional) independence is preserved in the limit under any of the convergence types \ref{ONC}--\ref{MC}--\ref{ASC}--\ref{STC}--\ref{HC}--\ref{SC}--\ref{WC}. For conditional independence, we must except \ref{WC}, cf. Example~\ref{example:indep} below. 
\end{remark}
\begin{proof}
We may assume $\II$ is finite; then, thanks to \eqref{preliminaries:weak}, last bullet point, via mathematical induction and properties of conditional independence, we reduce to the case $\II=\{0,1\}$; finally, on account of \eqref{preliminaries:weak}, second bullet point, there is no loss of generality in taking $\BB_0^i=\BB_\PP^i$ for each $i\in \II$. 
Now take  $A^0_0\in \BB_0^0$ and  $A^1_0\in \BB_0^1$ arbitrary. Then for each $i\in \{0,1\}$ we find a sequence $(A_n^i)_{n\in \mathbb{N}}$ with $A_n^i\in \BB^i_n$ for each $n\in \mathbb{N}$ and with $\lim_{n\to\infty}\PP(A_n^i\triangle A_0^i)=0$. Now, $\PP_{\CC_n}(A_0^0\cap A_0^1)\to \PP_{\CC_0}(A_0^0\cap A_0^1)$,  $\PP_{\CC_n}(A_0^0)\to \PP_{\CC_0}(A_0^0)$ and  $\PP_{\CC_n}(A_0^1)\to \PP_{\CC_0}(A_0^1)$ in $\PP$-probability. Also, for each $n\in \mathbb{N}$, 
\begin{itemize}
\item thanks to independence, $\PP$-a.s. $\PP_{\CC_n}(A_n^0\cap A_n^1)  =\PP_{\CC_n}(A_n^0)\PP_{\CC_n}(A_n^1)$,
\item whilst using the elementary equality $\vert \mathbbm{1}_C-\mathbbm{1}_D\vert= \mathbbm{1}_{C\triangle D}$ for sets $C$ and $D$, we find that $$\PP\vert \PP_{\CC_n}(A_0^0\cap A_0^1)-\PP_{\CC_n}(A_n^0\cap A_n^1)\vert$$ $$=\PP\vert \PP_{\CC_n}(A_0^0\cap A_0^1)-\PP_{\CC_n}(A_n^0\cap A_0^1)+\PP_{\CC_n}(A_n^0\cap A_0^1)-\PP_{\CC_n}(A_n^0\cap A_n^1)\vert$$ $$\leq \PP(A_0^0\triangle A_n^0)+\PP(A_0^1\triangle A_n^1)$$ and likewise $$\PP\vert \PP_{\CC_n}(A_n^0) \PP_{\CC_n}(A_n^1)-\PP_{\CC_n}(A_0^0) \PP_{\CC_n}(A_0^1)\vert\leq  \PP(A_0^0\triangle A_n^0)+\PP(A_0^1\triangle A_n^1).$$ 
\end{itemize}
In particular, $\PP_{\CC_n}(A_n^0) \PP_{\CC_n}(A_n^1)-\PP_{\CC_n}(A_0^0) \PP_{\CC_n}(A_0^1)\to 0$ and $\PP_{\CC_n}(A_0^0\cap A_0^1)-\PP_{\CC_n}(A_n^0)\PP_{\CC_n}(A_n^1)\to 0$ in $\PP$-probability. 
Since convergence in probability is preserved under addition and multiplication and since the limit in probability is a.s. unique, letting $n\to\infty$, yields that $\PP$-a.s. $\PP_\CC(A^0_0\cap A^1_0)=\PP_\CC(A^0_0)\PP_\CC(A^1_0)$, as required.
\end{proof}

\begin{example}\label{example:indep}
We show that conditional independence is generally not preserved under \ref{WC} to the $\PP\text{-}\liminf$ (see \eqref{preliminaries:weak}, second bullet point, for the notation). --- Without the latter insistence, a counterexample is trivial: take any $A\in \FF$ with $\PP(A)\in (0,1)$, for $n\in \mathbb{N}$ let $\BB_n=\HH:=\overline{\sigma(A)}^\PP$ be the $\PP$-complete $\sigma$-field generated by $A$. Then $\BB_n$ is conditionally independent of $\BB_n$ given $\BB_n$ for each $n$, but the strong (indeed, in every sense) limit $\HH$ is not conditionally independent of itself given the trivial $\sigma$-subfield, to which the $\BB_n$ converge weakly.) --- Take $\Omega=[0,1)$, $\FF=\mathcal{B}(\Omega)$ the Borel $\sigma$-field, $\PP=$ Lebesgue measure. For $n\in \mathbb{N}$ set $$B_n:=\bigcup_{k=0}^{2^{n-2}-1}\left[\frac{2k}{2^n},\frac{2k+1}{2^n}\right),$$ and let $\BB_n=\overline{\sigma(B_n)}^\PP$ be the $\PP$-complete $\sigma$-field generated by $B_n$. Finally let $\BB_0=\overline{\{\emptyset,\Omega\}}^\PP$ be the trivial $\sigma$-subfield. Then $\BB_n$ converges weakly to $\BB_0=\BB_\PP=\PP\text{-}\liminf_{n\to\infty}\BB_n$, but not strongly \cite[Example~3.4]{piccinini}. Finally, denote $x=(\sqrt{6}-1)/8$; notice that $1/2<1-2x+1/16<1-x<1-x+\frac{1}{16}<1$;  and consider the events $$A:=\left[\frac{3}{16},\frac{7}{16}\right)\cup \left[1-x,1\right)\text{ and }B:=\left[\frac{1}{16},\frac{5}{16}\right)\cup\left[1-2x+\frac{1}{16},1-x+\frac{1}{16}\right)$$ so that $$A\cap B=\left[\frac{3}{16},\frac{5}{16}\right)\cup \left[1-x,1-x+\frac{1}{16}\right).$$ A simple calculation yields that $\PP$-a.s. for each $n\in \mathbb{N}_{\geq 4}$,
$$\PP(A\vert \BB_n)=\PP(B\vert \BB_n)=\frac{1/8}{1/4}\mathbbm{1}_{B_n}+\frac{1/8+x}{3/4}\mathbbm{1}_{\Omega\backslash B_n}$$ whilst $$\PP(A\cap B\vert \BB_n)=\frac{1/16}{1/4}\mathbbm{1}_{B_n}+\frac{1/16+1/16}{3/4}\mathbbm{1}_{\Omega\backslash B_n}.$$ This renders the $\PP$-a.s. equality $\PP_{\BB_n}(\mathbbm{1}_{A})\PP_{\BB_n}(\mathbbm{1}_{B})=\PP_{\BB_n}(\mathbbm{1}_{A\cap B})$. Thus $A$ and $B$, equivalently the respective $\PP$-complete $\sigma$-fields generated by them, are conditionally independent given $\BB_n$ for each $n\in \mathbb{N}_{\geq 4}$. But $A$ and $B$ are not conditionally independent given the weak limit $\BB_0=\BB_\PP$ of the $\BB_n$, for they are not independent: $(1/4+x)^2\ne 1/8+1/16$, as is readily verified.\finish
\end{example}

\section{Invariance under passage to equivalent probability measure}\label{section:invariance}
In what follows, for a convergence mode (iC), by saying that it is invariant under passage to an equivalent probability measure, we mean that, whenever $(\BB_n)_{n\in \mathbb{N}_0}\subset \FFF$,

$\BB_n\to\BB_0$ in the sense of (iC) under $\PP$ $\iff$ $\BB_n\to\BB_0$ in the sense of (iC) under $\QQ$, 

\noindent provided $\QQ\sim \PP$. (Note that, given $(\Omega,\FF)$, $\FFF$ depends on $\PP$ only through $\NN=\PP^{-1}(\{0\})$.) 

Recall that for finite measures $\mu$ and $\nu$ on $(\Omega,\FF)$,  $\mu\ll\nu$ is equivalent to $$\forall \epsilon\in (0,\infty)\, \exists \delta\in (0,\infty)\, \forall A\in \FF:\, \nu(A)<\delta\Rightarrow \mu(A)<\epsilon.$$

\begin{proposition}\label{prop:two}
\emph{\ref{HC}} is invariant under passage to an equivalent probability measure. Indeed the distance $D$, up to equivalence, depends on $\PP$ only up to equivalence. 
\end{proposition}
\begin{remark}
By \eqref{preliminaries:order}, \ref{ONC}, $q<p$, is equivalent to \ref{HC}, so that \ref{ONC}, $q<p$, too is invariant under passage to an equivalent probability measure. The case $p=q$ remains open.
\end{remark}
\begin{proof}
Let $\QQ\sim \PP$, denote by $D_\PP$ and $D_\QQ$ the metrics associated to $\PP$ and $\QQ$, respectively. Let $\epsilon\in (0,\infty)$, $\AA\in \FFF$. Since $\QQ\ll\PP$, there is a $\delta\in (0,\infty)$ such that for all $A\in \FF$, $\PP(A)<\delta \Rightarrow \QQ(A)<\epsilon$. Then for all $\BB\in \FFF$, $D_\PP(\AA,\BB)<\delta$ implies $D_\QQ(\AA,\BB)\leq 2 \epsilon$. 
\end{proof}

\begin{proposition}
\emph{\ref{WC}} is invariant under passage to an equivalent probability measure. 
\end{proposition}
\begin{proof}
By \eqref{preliminaries:weak}, second bullet point, it is enough to verify that $\BB_\PP=\BB_\QQ$ whenever $\PP\sim\QQ$.  
But assuming $\QQ\ll\PP$, thanks to the equivalent condition for absolute continuity noted above, if for some $A\in \FF$ there exist $A_n\in \BB_n$ for $n\in \mathbb{N}$ with $\lim_{n\to\infty}\PP(A_n\triangle A)=0$, then also $\lim_{n\to\infty}\QQ(A_n\triangle A)=0$. 
\end{proof}
Recall now the statement of abstract Bayes' theorem. Letting $\QQ$ be another probability measure on $\FF$, equivalent to $\PP$: 
\begin{quote}
For any $\FF$/$\mathcal{B}([-\infty,\infty])$-measurable $X$ and any sub-$\sigma$-field $\GG$ of $\FF$, $\QQ_ \GG (X) \PP_\GG(\frac{d\QQ}{d\PP})= \PP_\GG (\frac{d\QQ}{d\PP} X)$ a.s., in the sense that the left hand-side is well-defined iff the right hand-side is so, whence they are equal.
\end{quote}
\begin{proposition}
\emph{\ref{SC}} is invariant under passage to an equivalent probability measure. 
\end{proposition}
\begin{proof}
Assume $\BB_n\to\BB_0$ strongly under $\PP$. Let $\QQ\sim \PP$. Note $d\QQ/d\PP$ and then all the $\PP_{\BB_n}(\frac{d\QQ}{d\PP})$, $n\in \mathbb{N}_0$, may be chosen from their equivalence classes to be strictly positive everywhere. Let $A\in \FF$. Then by Bayes' rule a.s. $$\QQ_{\BB_n}\mathbbm{1}_A=\PP_{\BB_n}\left(\frac{d\QQ}{d\PP} \mathbbm {1}_A\right)\Big/\PP_{\BB_n}\left(\frac{d\QQ}{d\PP}\right).$$ By \eqref{preliminaries:strong} the numerator and denominator both converge as $n\to\infty$ in  $\LL^1(\PP)$, hence in $\PP$- (equivalently, $\QQ$-) probability, to the respective expressions in which $\BB_0$ replaces $\BB_n$. Convergence in probability is preserved under taking quotients (assuming the denominators are non-zero; e.g. from the characterization through the a.s. convergence of subsequences) and the claim follows by another application of Bayes' rule. 
\end{proof}

\begin{remark}
According to \cite[Proposition~3.3, Lemma~1.3]{alonso-brambila} \ref{SC} is equivalent to the conjunction of \ref{WC} and 
\begin{enumerate}[label=($\perp$C),ref=($\perp$C)]
\item\label{OC} \emph{Orthogonal convergence}. $\BB_n$ converges to $\BB_0$  \emph{orthogonally} if $\mathbbm{1}_{A_n}-\PP_{\BB_0}\mathbbm{1}_{A_n}\to 0$ weakly in $\LL^2(\PP)$, whenever $A_n\in \BB_n$ for each $n\in \mathbb{N}$.
\end{enumerate}
It remains open whether \ref{OC} too is invariant under passage to an equivalent probability measure. 
\end{remark}

\begin{proposition}
\emph{\ref{ASC}} is invariant under passage to an equivalent probability measure. 
\end{proposition}
\begin{proof}
Assume $\BB_n\to\BB_0$ in the almost-sure sense under $\PP$. Let $\QQ\sim \PP$. Then by Bayes' rule, for any $f\in \LL^1(\QQ)$, a.s. $$\QQ_{\BB_n}f=\PP_{\BB_n}\left(\frac{d\QQ}{d\PP} f\right)\Big/\PP_{\BB_n}\left(\frac{d\QQ}{d\PP}\right).$$ Since $\frac{d\QQ}{d\PP} f$ and $\frac{d\QQ}{d\PP}$ both belong to $\LL^1(\PP)$, by the very definition of \ref{ASC}, the numerator and denominator both converge $\PP$- (equivalently $\QQ$-) a.s. as $n\to\infty$ to the respective expressions in which $\BB_0$ replaces $\BB_n$. Another application of Bayes' rule concludes the argument. 
\end{proof}

\begin{question}
Given this invariance of the various convergence modes, can something akin to the characterization of convergence in probability through the a.s. convergence of subsequences, be offered? I.e. can the convergence modes be characterized in such a way as to make manifest the invariance under passage to an equivalent probability measure?
\end{question}

\bibliographystyle{amsplain}
\bibliography{Biblio_convergence}
\end{document}